\documentclass[12pt]{amsart}

\allowdisplaybreaks
\usepackage{mathpazo}
\usepackage{amsmath,amsfonts,amssymb}
\usepackage{amsrefs}
\usepackage{amsthm}
\usepackage{latexsym,amsmath,amssymb,amsfonts}
\usepackage{rotating}
\usepackage{mathrsfs}
\usepackage{xypic} \xyoption{all}
\usepackage{amscd}
\usepackage{hyperref} 
\usepackage{euscript}
\usepackage{hhline}
\usepackage{graphicx,epstopdf}
\usepackage{epsfig}
\usepackage{xcolor}
\usepackage{textcomp}
\usepackage[all,color]{xy}
\usepackage{spectralsequences}
\usepackage{sseq}

\newlength{\fighskip} \fighskip=2pt
\newlength{\figvskip} \figvskip=3pt

\usepackage{hyperref}
\newcommand*{\figbox}[2]{{
  \def\figscale{#1}
  \def\arraystretch{0.8}
  \arraycolsep=0pt
  \begin{array}{c}
    \vbox{\vskip\figscale\figvskip
      \hbox{\hskip\figscale\fighskip
        \includegraphics[scale=\figscale]{#2}}}
  \end{array}}}

\advance\textheight by \topskip
\numberwithin{equation}{section}

\newcommand{\C}{\mathbb{C}}

\newcommand{\R}{\mathbb{R}}

\newcommand{\W}{\mathcal{W}}

\DeclareMathOperator{\Sym}{Sym}

\DeclareMathOperator{\Tr}{Tr}

\newcommand{\A}{\mathcal A}

\renewcommand{\W}{\mathcal W}

\theoremstyle{plain}
\newtheorem{thm}{Theorem}[section]
\newtheorem{thm-defn}{Theorem/Definition}[section]
\newtheorem{lem}[thm]{Lemma}
\newtheorem{lem-defn}[thm]{Lemma/Definition}
\newtheorem{prop}[thm]{Proposition}
\newtheorem{cor}[thm]{Corollary}

\theoremstyle{definition}
\newtheorem{defn}[thm]{Definition}
\newtheorem{notn}[thm]{Notation}
\newtheorem{eg}[thm]{Example}

\theoremstyle{remark}
\newtheorem{rmk}[thm]{Remark}

\allowdisplaybreaks[4]  

\begin{document}
\title[Formal Deformation quantization as a Fr\'echet algebra]{Formal Deformation quantization as a Fr\'echet algebra}

\author[Li]{Qin Li}
\address{Shenzhen Institute for Quantum Science and Engineering, Southern University of Science and Technology, Shenzhen, China}
\email{liqin@sustech.edu.cn}

\thanks{}

\begin{abstract}
We define a Fr\'echet topology on the space  $C^\infty(X)[[\hbar]]$ of formal smooth functions on a symplectic manifold $X$,  by constructing a sequence of semi-norms on it. For any star product $\star$ on $C^\infty(X)[[\hbar]]$ making it a formal deformation quantization of $X$, we will show that the quantum product $\star$ is jointly continuous, and making it a Fr\'echet algebra. We will show a quantum Weierstrass theorem which says quantum polynomials are locally dense in all formal smooth functions. We will also show that the canonical trace of any formal deformation quantization is continuous under this Fr\'echet topology. 
\end{abstract}

\maketitle

\tableofcontents

\section{Introduction}

A symplectic manifold $X$ is the mathematical description of the phase space of a classical mechanical system, whose classical observable algebra is modeled by smooth functions $C^\infty(X)$. Deformation quantization of $C^\infty(X)$ models the {\em quantum} observable algebra of the corresponding quantum mechanical system. There are two types of deformation quantization of a symplectic manifold: strict and formal deformation. 

A {\em formal} deformation quantization of a symplectic (K\"ahler) manifold denotes a star product $\star$ on the space $C^\infty(X)[[\hbar]]$ which deforms the classical commutative product on smooth functions as 
\begin{equation}\label{equation: hbar-power-expansion}
f\star g=fg+\sum_{i\geq 1}\hbar^i\cdot C_i(f,g).
\end{equation}
It is required that $C_1(f,g)-C_1(g,f)=\{f,g\}$, which implies that the first order non-commutativity of $\star$ is the Poisson bracket. For general $C_i(-,-)$'s, they are usually required to be bi-differential operators, as we do in this paper. The formal variable $\hbar$ corresponds to Planck's constant in physics. From a physical point view, such a star product only works when convergence of the series in equation \eqref{equation: hbar-power-expansion} is satisfied for the evaluation of $\hbar$ at the Planck's constant.
On the other hand, the {\em strict} deformation quantization introduced by Rieffel in \cite{Rieffel1989} uses the framework of $C^*$ algebras. Instead of a formal variable, here $\hbar$ belongs to an additional  $\mathbb{R}$ parametrizing star products $\star_\hbar$ which are required to depend continuously on $\hbar$.  

It is believed that there are close relations between these two types of deformation quantizations.  For instance, the formal star product should be the asymptotic expansion of the product of $C^*$ algebras in strict quantization. Conversely, subalgebras of $C^\infty(X)$ for which the power series \eqref{equation: hbar-power-expansion} converges for small values $\hbar>0$ can recover (part of) the information of strict quantization. To understand this relation, Waldmann et al. introduced in \cite{Beiser-Romer-Waldmann, Beiser-Waldmann}  a Fr\'echet algebraic deformation quantization which is intermediate between the formal and strict quantization, and considered a Fr\'echet topology on certain subclass of functions on which the formal star product is convergent for small $\hbar>0$. See also \cite{OMMY} for the study of convergence of formal deformation quantization.

Although the convergence issue for formal deformation quantization of smooth functions and strict quantization are  important questions, it is also essential to consider quantization of formal functions $C^\infty(X)[[\hbar]]$. Here we describe an example in geometric quantization. Recall that the Hilbert spaces in the Kostant-Souriau geometric quantization of K\"ahler manifolds are given by holomorphic sections of the (tensor powers of) pre-quantum line bundles $L$:
$$
\mathcal{H}_k:=H^0_{\bar{\partial}}(X, L^{\otimes k}).
$$
For a subclass of functions with a formal variable $\hbar$, we can {\em quantize} them to holomorphic differential operators on the Hilbert space $\mathcal{H}_k$ where $\hbar$ takes the evaluation $\hbar=1/k$. An explicit example is as follows: let $f$ denote the moment map corresponding to a Hamiltonian vector field $V$ on $X$ which also preserves the complex structure on $X$, then the formal function $f+\hbar\cdot\Delta f$ acts on $\mathcal{H}_k$ for {\em all} $k>0$ via the associated Kostant-Souriau operator. Thus it is necessary to keep this formal variable $\hbar$. In the K\"ahler geometry setting, a general notion of {\em quantizable} (formal) functions is defined in \cite{ChaLeuLi2023}, which act on the Hilbert spaces as holomorphic differential operators. In this paper, we will define quantizable functions on real symplectic manifolds in a similar way which also acts on geometric quantization when there is a polarization.  These functions are natural generalization of {\em quantum polynomials} defined via a Quantum Darboux Theorem in \cite{Fed1996}.  There are the following inclusions 
$$
\textit{Quantum polynomials}\subset\textit{Quantizable functions}\subset\textit{Formal smooth functions}.
$$
 In particular, quantizable functions (on both real symplectic and K\"ahler manifolds) are polynomials in the formal variable $\hbar$, and so is the product of them, and thus are automatically convergent with respect to any evaluation of $\hbar$. 

It is natural to ask if there are enough such quantizable functions, which is one of the basic requirements for quantization of functions. For this, we first construct in this paper  a Fr\'echet topology on $C^\infty(X)[[\hbar]]$ which involves the formal variable and all smooth functions (instead of a subclass of functions). This topology is defined via a sequence of semi-norms induced by the supremum of derivatives of functions (Definition \ref{definition: Frechet-space-semi-norms}). These semi-norms are indexed by a degree involving both the order of derivatives and the polynomial degree of  $\hbar$ (see Lemma/Definition \ref{lemma-definition: semi-norms-formal-smooth-functions}). We will show that any star product on $C^\infty(X)[[\hbar]]$ is jointly continuous with respect to this Fr\'echet topology, and making it a Fr\'echet algebra. 

We will prove the {\em Quantum Weierstrass Approximation} which says that locally on $X$, quantum polynomials are dense in all formal smooth functions under the Fr\'echet topology. This implies that there are {\em enough} quantizable functions: 
\begin{thm}(Theorem \ref{theorem: quantum-Weierstrass})
For any Darboux chart $U\subset X$, the space of quantizable functions are dense in all formal smooth functions with respect to the Fr\'echet topology. 
\end{thm}

For a formal deformation quantization $(C^\infty(X)[[\hbar]],\star)$, the most important algebraic invariant is the {\em trace} of these formal functions which is uniquely defined (see Definition \ref{definition: trace-map}).  To justify that this Fr\'echet topology on $C^\infty(X)[[\hbar]]$ is appropriate, we will show the continuity of the (renormalized) trace with respect to the Fr\'echet topology by applying the Batalin-Vilkovisky quantization method developed in \cite{GraLiLi2017}.  
\begin{thm} (Theorem \ref{theorem: continuity-trace-map})
The renormalized trace map $\widetilde{\text{Tr}}$ for a deformation quantization $(C^\infty(X)[[\hbar]],\star)$  is continuous with respect to the Fr\'echet topologies on both sides. 
\end{thm}
The trace of a formal smooth function $f\in C^\infty(X)[[\hbar]]$ is identified with the correlation function of a one-dimensional Chern-Simons model, and the proof is based on explicit computation via Feynman graph weights and the integral of a differential form, which at each point $x_0\in X$ depend on the jets of $f$ and the curvature (and its covariant derivatives) of $X$. 

\subsection*{Notations}
	\begin{itemize}
		\item Throughout this paper, we let $(X, \omega)$ denote a symplectic manifold of real dimension $2n$ with $\omega$ its symplectic form. 
		\item The fiberwise Moyal-Weyl product on the formal Weyl bundle is denoted by $\star_{MW}$.
		\item Formal star products on symplectic manifolds are denoted by $\star$. 
		\item We let $\A_X^{\bullet}$ denote differential forms on a manifold $X$, and let $TX$ and $T^*X$ denote the real tangent and cotangent bundle on $X$ respectively. 
	\end{itemize}

\section*{Acknowledgement}
The author is supported by National Science Foundation of China (Project No. 12471061). The author would like to thank Kwokwai Chan, Conan Leung, Yu Qiao and Tony Yutung Yau for many helpful discussions.

\section{Deformation quantization as a Fr\'echet algebra}

There have been tremendous studies on deformation quantization of symplectic manifolds, most of which focus on its geometric construction (\cite{Fed1994}),  algebraic properties  including the trace of these quantum algebras (\cite{Fed1996, Nest-Tsygan-1995, Fed2000, FFS1995}), or its description as certain operators on Hilbert spaces, for instance  deformation quantization on K\"ahler (and also symplectic) manifolds via Berezin-Toeplitz operators. There are also studies on the covergence of these quantum algebras for small enough evaluation of $\hbar$ (\cite{OMMY}). Here we still regard $\hbar$ as a formal variable, and  define a Fr\'echet topology on $C^\infty(X)[[\hbar]]$ which is compatible with any star product $\star$ making $(C^\infty(X)[[\hbar]],\star)$ a formal deformation quantization.

This Fr\'echet topology on $C^\infty(X)[[\hbar]]$ is defined via a sequence of semi-norms $\{||-||_{\hbar,k}:k\geq 0\}$ which are roughly given by a sum of the suprema of derivatives of a fixed order. It turns out that for any star product $\star$ on $C^\infty(X)[[\hbar]]$, the multiplication operation $(f,g)\mapsto f\star g$ is jointly continuous thus making $(C^\infty(X)[[\hbar]],\star)$ a Fr\'echet algebra. 

\subsection{Fr\'echet spaces}
\

There are several equivalent definition of Fr\'echet spaces, we will adopt the approach via a countable family of semi-norms which is convenient for the space of formal smooth functions. We first recall the notion of semi-norms:
\begin{defn}\label{definition: semi-norms}
	A non-negative function $p$ on a vector space $V$ satisfying the following conditions is called a {\em semi-norm}:
	\begin{enumerate}
		\item For any scalar $\lambda$ and $v\in V$, there is $p(\lambda\cdot v)=|\lambda|\cdot p(v)$;
		\item For all $v_1,v_2\in V$, there is $p(v_1+v_2)\leq p(v_1)+p(v_2)$;
		\item For any $v\in V$, there is $p(v)\geq 0$. 
	\end{enumerate}
\end{defn}
A simple observation is that suppose $p_1, p_2$ are two semi-norms, there sum $p_1+p_2$ is also a semi-norm. 
\begin{defn}\label{definition: Frechet-space-semi-norms}
	Suppose on a vector space $V$, there are a countable family of seminorms $\{||\cdot||_k\}_{k\in\mathbb{Z}_{\geq 0}}$ satisfying the following two properties:	
	\begin{itemize}
		\item Suppose $x\in V$, and $||x||_k=0$ for all $k\geq 0$, then $x=0$;
		\item If a sequence $\{x_n\}$ in $V$ which is Cauchy with respect to each semi-norm $||\cdot||_k$, then there exists $x\in V$, such that $\{x_n\}$ converges to $x$ with respect to each semi-norm $||\cdot||_k$. 
	\end{itemize}
    The associated Fr\'echet topology on $V$ is defined via the following translation-invariant complete metric: 
    \begin{equation}\label{equation: distance-in-Frechet-topology}
    	d(x,y):=\sum_{k=0}^\infty 2^{-k}\cdot\frac{||x-y||_k}{1+||x-y||_k}.
    \end{equation}
	Equivalently, a Fr\'echet space is a locally convex metrizable topological vector space that is complete. 
\end{defn}

	Loosely speaking, the definition says that the distance $d(x,y)$ will be "small enough" if there exists $K>>0$, such that $||x-y||_l$ is small enough for $l=0,1,\cdots, K$.  For instance,  to construct a convergence sequence in $V$ under this topology, we can ignore semi-norms $||\cdot||_l$ for $l>>0$ for each fixed $\epsilon>0$ in  the estimate. This is similar to the asymptotic analysis where $O(\hbar^l)$ is also ignored.

\begin{eg}\label{example: sequence-complex-numbers}
	Let $\C^{\mathbb{N}}$ denote the space of complex number valued sequences. We can define an increasing family of semi-norms on $\C^{\mathbb{N}}$ by 
	\begin{equation}\label{equation: semi-norms-sequence}
		||\varphi||_n:=\max_{k\leq n}|\varphi(k)|. 
	\end{equation} 
\end{eg}
\begin{eg}\label{example: smooth-function-Frechet-space}
For any $m\geq 1$, the space $C^\infty(\mathbb{R}^m)$ of infinitly differentiable functions on $\mathbb{R}^m$ is a Fr\'echet space. For any non-negative integer $l \geq 0$, we can define a semi-norm for every $f\in C^\infty(\mathbb{R}^m)$:
\begin{equation}\label{equation: semi-norms-smooth-functions}
||f||_{l}:=\sup_{|I|\leq l, ||x_0||\leq l}\bigg|\frac{\partial^{|I|}f}{\partial x^I}(x_0)\bigg|.
\end{equation}
Here $I$ denotes a multi-index. 
It is easy to see that these semi-norms satisfy the conditions in Definition \ref{definition: Frechet-space-semi-norms} and makes $C^\infty(\mathbb{R}^m)$ a Fr\'echet space. Since both the domain of taking supremum and the order of partial derivatives increase as $l$ grows, these semi-norms are increasing: $||-||_0\leq ||-||_1\leq\cdots $.
\end{eg}
\begin{rmk}
	In the definition of the $l$-th semi-norms on $C^\infty(\mathbb{R}^n)$ (and later for general smooth manifolds), we took the supremum of all partial derivatives of order $\leq l$ instead of the more standard supremum of partial derivatives of exactly order $l$. This does NOT affect the resulting Fr\'echet topology, but will be more convenient for the discussion of the joint continuity of multiplication in later sections.
\end{rmk}
\begin{rmk}
In section \ref{subsection: example-smooth-functions-manifolds}, we will describe the Fr\'echet topology on the space of smooth functions on a smooth manifold $X$ (compact or not). 
\end{rmk}

\subsection{Fr\'echet algebra}
\

We first give the definition of Fr\'echet algebras:
\begin{defn}
A Fr\'echet algebra is an associative algebra $A$ over the real or complex numbers that at the same time it is also a  Fr\'echet space, and that the multiplication operation $(a,b)\mapsto a\star b$ is required to be {\em jointly continuous}, which requires that for two convergent sequences $a_k\rightarrow a$ and $b_k\rightarrow b$,  there is the convergence of the product sequence $a_k\star b_k\rightarrow a\star b$. 
\end{defn}

An equivalent description of the joint continuity of the muliplication is described in the following lemma:
\begin{lem}\label{lemma: equivalent-description-Frechet-algebra}
If $\{||\cdot||_k: k\geq 0\}$ is an increasing family of seminorms for the Fr\'echet topology of $A$ (meaning that $||\cdot||_1\leq ||\cdot||_2\leq\cdots$). The joint continuity of multiplication is equivalent to there being a constant
$C_n$ and integer $m\geq n$, such that 
\begin{equation}\label{equation: inequality-Frechet-algebra}
||a\star b||_n\leq C_n\cdot ||a||_m||b||_m.
\end{equation}
If the natural number $m$ in the above definition can be chosen to be $m=n$, then such Fr\'echet algebras are called $m$-convex. 
\end{lem}
\begin{eg}\label{example: C-hbar-Frechet-algebra}
In Example \ref{example: sequence-complex-numbers}, we defined semi-norms on $\C^{\mathbb{N}}\cong\C[[\hbar]]$ making it a Fr\'echet space. It is known that these semi-norms makes $\C^{\mathbb{N}}$ an $m$-convex Fr\'echet algebra with respect to the pointwise multiplication. 

There is also another natural multiplication $\star$ on $\C[[\hbar]]$. We will show that $\mathbb{C}[[\hbar]]$ is also an $m$-convex Fr\'echet algebra under this product. 
Given $a,b\in A$, let $c:=a\star b$ with $c_m=\sum_{i+j=m}a_i\cdot b_j$, and there is
\begin{align*}
	 ||c||_k=&\max_{m\leq k}|c_m|\\
	 \leq & C_k\cdot ||a||_k\cdot ||b||_k.
\end{align*}
\end{eg}



\subsubsection{Example: smooth functions on manifolds}\label{subsection: example-smooth-functions-manifolds}
\

In this subsection, we will explain that smooth functions on a smooth manifold $X$ is also an example of Fr\'echet algebras. We will choose a countable covering of $X=\cup_{i\geq 0} U_i$ (we are not assuming that $X$ is compact) with a fixed coordinate system on each $U_i$, such that each finite union has a compact closure. For each $l\geq 0$, we define the $l$-th semi-norm as follows: let $f\in C^\infty(X)$:
\begin{equation}\label{equation: semi-norm-function-manifolds}
	||f||_l:=\sup_{x_0\in\cup_{i\leq l}\overline{U_i}, |I|\leq l}\bigg\{\bigg|\frac{\partial^{|I|} f}{\partial x^I}(x_0)\bigg|\bigg\}.
\end{equation}
Here we are taking the supremum of all the partial derivatives of order $\leq l$ on all the closure $\overline{U_i}$'s with $i\leq l$, under the chosen local coordinates. Since as $l$ grows, the order of the partial derivatives increases and the domain where we take these derivatives becomes larger, we actually obtain an increasing family of semi-norms. 
\begin{prop}\label{proposition: smooth-function-Frechet-algebra}
	The space of smooth functions $C^\infty(X)$ on a smooth manifold (compact or non-compact) forms a Fr\'echet algebra with respect to the semi-norms $||-||_l$. 
\end{prop}
\begin{proof}
	We will apply Lemma \ref{lemma: equivalent-description-Frechet-algebra} here. 
	Since the semi-norms $||-||_{l}$ are defined as the supremum of partial derivatives over enlarging domains, it is enough to prove a pointwise estimate to show the inequality  \eqref{equation: inequality-Frechet-algebra}. 
		For any multi-index $I$ and any $x_0\in\cup_{i\leq |I|} U_i$, there is by Leibniz's rule
		\begin{align*}
			\bigg|\frac{\partial^{|I|}(fg)}{\partial x^I}(x_0)\bigg|=&\bigg|\sum_{K+J=I}C_{K,J}\cdot\frac{\partial^{|K|}f}{\partial x^K}\frac{\partial^{|J|}g}{\partial x^J}\bigg|_{x_0}\\
			\leq & C_I\cdot ||f||_{|I|}\cdot ||g||_{|I|}.
		\end{align*}
		Here the constant $C_I$ only depends on the multi-index $I$. 
		It follows that $||f\cdot g||_l\leq C_l\cdot ||f||_l\cdot ||g||_l$ for any $l\geq 0$.  
\end{proof}

\begin{rmk}
	Different choices of a covering of a smooth manifold $X$ by coordinate charts induce different semi-norms, but will give the same Frechet topology on $C^\infty(X)$. 
\end{rmk}
\begin{rmk}\label{remark: estimate-semi-norms-local}
	For a compact symplectic manifold $X$, we will fix a finite covering  $X=\cup_i U_i$ by Darboux coordinate charts. 
\end{rmk}


\subsection{Formal deformation quantization as a Fr\'echet algebra}
\

Recall that a {\em formal deformation quantization} of a symplectic manifold $(X, \omega)$ is a formal deformation of the commutative algebra $(C^{\infty }(X),\cdot)$ equipped with the classical pointwise multiplication to a noncommutative one $( C^{\infty }( X)[[\hbar]] ,\star) $ equipped with a {\em star product} $\star$ of the following form
$$
f\star g=fg+\sum_{i\geq 1}\hbar^i\cdot C_i(f,g),
$$
where each $C_i(-,-)$ is a bi-differential operator of degree at most $i$ to each function, and that the leading order of noncommutativity is the Poisson bracket $\{-,-\}$ associated to $\omega$, i.e.,
\begin{equation}\label{equation: Poisson-bracket}
	C_1(f,g)-C_1(g,f)= \left. \frac{d}{d\hbar}\left( f\star g - g\star f\right)\right\vert_{\hbar=0}
	= \{ f,g \}.
\end{equation}

In this subsection, we will define a sequence of increasing semi-norms on the space of formal smooth functions $C^\infty(X)[[\hbar]]$ which makes it a Fr\'echet algebra with respect to any star product $\star$ making $C^\infty(X)[[\hbar]]$ a formal deformation quantization of $X$.


\begin{lem-defn}\label{lemma-definition: semi-norms-formal-smooth-functions}
	Let $f=f_0+\hbar\cdot f_1+\hbar^2\cdot f_2+\cdots\in C^\infty(X)[[\hbar]]$ be a formal smooth function. For every $k\geq 0$, the semi-norm $||\cdot||_{\hbar,k}$ is defined as 
	\begin{equation}\label{equation: k-semi-norm-formal-function}
		||f||_{\hbar,k}:=\sum_{i+j=k}||f_i||_j.
	\end{equation}
	Here $||\cdot||_j$ denotes the $j$-th semi-norm of smooth functions on smooth manifolds in equation \eqref{equation: semi-norm-function-manifolds}. There is $||\cdot||_{\hbar,0}\leq ||\cdot||_{\hbar,1}\leq\cdots$, and thus they form a family of increasing semi-norms. There is the following additive formula for the semi-norms $||-||_{\hbar,l}$'s:
	\begin{equation}\label{equation: additive-formula-semi-norms}
	||f||_{\hbar,l}=\sum_{i\geq 0}||\hbar^i\cdot f_i||_{\hbar,l}
	\end{equation}
	(Note that this is a finite sum for every $l$ by Remark \ref{remark: semi-norm-depends-on-truncation-formal-function}.) 
\end{lem-defn}
\begin{rmk}\label{remark: semi-norm-depends-on-truncation-formal-function}
	According to the definition of these semi-norms, higher order terms in $\hbar$ power will not contribute to lower order semi-norms. Explicitly, for a formal smooth function $f=f_0+\hbar\cdot f_1+\hbar^2\cdot f_2+\cdots$, the semi-norm $||f||_{\hbar,k}$ is determined by its truncation $f_0+\hbar\cdot f_1+\cdots +\hbar^k\cdot f_k$. 
\end{rmk}

\begin{rmk}
We can interprete this Fr\'echet topology on $C^\infty(X)[[\hbar]]$ as a combination of the natural topology on $C^\infty(X)$ and the semi-norms on $\C[[\hbar]]$ in Example \ref{example: sequence-complex-numbers}. 
\end{rmk}

\begin{lem}\label{lemma: k-norms-define-Frechet-topology}
	The semi-norms defined in equation  \eqref{equation: k-semi-norm-formal-function} defines a Fr\'echet topology on $C^\infty(X)[[\hbar]]$. 
\end{lem}
\begin{proof}
	Let $f=f_0+\hbar\cdot f_1+\cdots$ denote a formal smooth function on $X$, whose semi-norms all vanish: $||f||_{\hbar,l}=0, l\geq 0$. Then $||f||_{\hbar,0}=0$ implies that the $f_0$ vanishes when restricted to $U_0$, i.e.,  $f_0|_{U_0}\equiv 0$. Since all these $U_i$'s cover $X$, it follows that $f_0\equiv 0$ on $X$.  Iteratively we can see that $f_m=0$ for all $m\geq 0$. 
	
	The statement that $C^\infty(X)[[\hbar]]$ is complete with respect to these semi-norms follows from the fact that $C^\infty(X)$ is complete with respect to the semi-norms $||\cdot||_k$ for every $k\geq 0$. 
\end{proof}
\begin{thm}\label{theorem: formal-functions-Frechet-algebra}
	A formal deformation quantization $(C^\infty(X)[[\hbar]],\star)$ of a symplectic manifold $X$ is a Fr\'echet algebra under the semi-norms $||\cdot ||_{\hbar,k}$. 
\end{thm}
\begin{proof}
	Let $f=f_0+\hbar\cdot f_1+\hbar^2\cdot f_2+\cdots$ and $g=g_0+\hbar\cdot g_1+\hbar^2\cdot g_2+\cdots$ denote two formal smooth functions on the symplectic manifold $X$. We write their star product as 
	\begin{align*}
	h=&h_0+\hbar\cdot h_1+\hbar\cdot h_2+\cdots\\
	:=&f\star g\\
	=&\sum_{m\geq 0}\sum_{m_1+m_2+m_3=m}\hbar^m\cdot C_{m_3}(f_{m_1}, g_{m_2})\\
	=&\sum_{m\geq 0}\sum_{m_1+m_2+m_3=m}\hbar^m\sum_{|I|,|J|\leq m_3}\alpha_{m_3,I,J}\frac{\partial^{|I|}f_{m_1}}{\partial x^I}\frac{\partial^{|J|}g_{m_2}}{\partial x^J}.
	\end{align*}
	As explained in Remark \ref{remark: semi-norm-depends-on-truncation-formal-function}, for any $l\geq 0$ only the truncation $$h_0+\hbar\cdot h_1+\cdots+\hbar^l\cdot h_l$$
	 of $h$ will contribute non-trivially to  $||h||_{\hbar,l}$.  Similar to the proof of Proposition \ref{proposition: smooth-function-Frechet-algebra}, we only need to do a pointwise estimate to prove the inequality \eqref{equation: inequality-Frechet-algebra}.  For any $m\leq l$, we let $K$ denote a multi-index with $|K|\leq l-m$.  For any $x_0\in \cup_{i\leq l}U_i$, there is then the following estimate (recall that $m=m_1+m_2+m_3$): 
	\begin{align*}
    \bigg|\frac{\partial^{|K|} C_{m_3}(f_{m_1},g_{m_2})}{\partial x^K}(x_0)\bigg|=&\bigg|\sum_{|I|,|J|\leq m_3}\frac{\partial^{|K|}}{\partial x^K}\left(\alpha_{m_3,I,J}\frac{\partial^{|I|}f_{m_1}}{\partial x^I}\frac{\partial^{|J|}g_{m_2}}{\partial x^J}\right)(x_0)\bigg|\\
    \leq &\ \tilde{C}_l\cdot ||f||_{\hbar,l}\cdot ||g||_{\hbar,l}.
	\end{align*}
    Here we are using the inequality $|I|,|J|\leq m_3\leq m\leq l$ and the inequality $|K|\leq l-m$ in the counting of the order of derivatives and indices of semi-norms. For instance,
    $$
    |I|+m_1+|K|\leq m_3+m_1+l-m\leq l. 
    $$
     The constant $\tilde{C}_l$ depends on $l$ and the supremum of the coefficients $\alpha_{m_3, I,J}$ and its derivatives of order $\leq l$. It then follows easily that for every $l\geq 0$, there exists a contant $C_l$, such that for all formal smooth functions $f, g$, there is 
    $$
    ||f\star g||_{\hbar,l}\leq C_l\cdot ||f||_{\hbar,l}\cdot ||g||_{\hbar,l}. 
    $$
\end{proof}
\begin{notn}\label{notation: metric-compact-subset}
This Fr\'echet topology on $C^\infty(X)[[\hbar]]$ is defined via the metric in equation \eqref{equation: distance-in-Frechet-topology}. For any compact subset $K\subset X$ and any formal smooth function on $X$, we can define semi-norms $||-||_{K,\hbar,l}$'s by taking supremum within $K$ instead of $X$ and obtain a metric denoted by $d_K(-,-)$. 
\end{notn}

\section{Fedosov quantization and Quantum Weierstrass Approximation Theorem}

In this section, we will first review Fedosov's geometric construction of deformation quantization on symplectic manifolds. By using a quantum Darboux Theorem, we define {\em quantum polynomials} on $X$, and show that they are dense in all smooth formal functions under the Fr\'echet topology. We also give some computation and estimates in Fedosov quantization in Section \ref{subsection: estimates-Fedosov-quantization} which we will need in Section \ref{section: continuity-trace-map}. 


\subsection{Fedosov's approach to deformation quantization}
\

We first briefly recall Fedosov's  beautiful geometric construction of deformation quantization on symplectic manifolds in \cite{Fed1994}. Let $(X,\omega)$ denote a symplectic manifold, on which there is the Weyl bundle:
$$
\W_{X,\mathbb{R}}:=\widehat{\Sym}TX^*.
$$ 
We choose local coordinates $(x^1,\cdots, x^{2n})$ on $X$ under which the symplectic form is locally 
$$
\omega=\omega_{ij}\cdot dx^i\wedge dx^j. 
$$
Let $y^i$ denote a local section of $TX^*$ corresponding to $x^i$ (thus is also a generator of the Weyl bundle $\W_{X,\R}$). Let $\A_X^\bullet(\W_{X,\R})$ denote differential forms on $X$ valued in $\W_{X,\R}$, whose section is locally given by a sum of the following terms:
\begin{equation}\label{equation: Weyl-bundle-monomial}
a_{k,i_1\cdots i_p,j_1\cdots j_q}\hbar^k\cdot dx^{j_1}\cdots dx^{j_q}\otimes y^{i_1}\cdots y^{i_p},
\end{equation}
where $a_{k,i_1\cdots i_p,j_1\cdots j_q}$ is symmetric in $i_1,\cdots, i_p$ and anti-symmetric in $j_1, \cdots, j_q$. 
On the (infinite rank) vector bundle $\W_{X,\R}[[\hbar]]$, there exists a fiberwise (quantum) Moyal-Weyl product which we denote by $\star_{MW}$, making it a bundle of algebras:
\begin{equation}\label{equation: Moyal-Weyl-product}
(\alpha\star_{MW}\beta)(x,\hbar):=\exp\left(\frac{\hbar}{2}\omega^{ij}\frac{\partial}{\partial y^i}\frac{\partial}{\partial z^j}\right)\alpha(y,\hbar)\beta(z,\hbar)\bigg|_{y=z=x}.
\end{equation}

We define two operators on $\A_X^\bullet(\W_{X,\R})$ by
\begin{equation}\label{equation: delta-operator}
	\begin{aligned}
	\delta(a):=&dx^i\wedge\frac{\partial a}{\partial y^i}=\frac{1}{\hbar}[\omega_{ij}(dx^i\otimes y^j-dx^j\otimes y^i),a]_{\star_{MW}},\\
	\delta^*(a):=&y^i\cdot\iota_{\partial_{x^i}}(a). 
	\end{aligned}
\end{equation}
We define $\delta^{-1}$ by requiring that it acts on the monomial \eqref{equation: Weyl-bundle-monomial} as $\frac{\delta^*}{p+q}$ if $p+q>0$ and zero if $p+q=0$. 

Fedosov constructed flat connections on the formal Weyl bundle $\W_{X,\R}[[\hbar]]$ of the following form:
\begin{equation}\label{equation: Fedosov-connection}
\begin{aligned}
D:=&\nabla-\delta+\frac{1}{\hbar}[I,-]_{\star_{MW}}\\
=&\nabla+\frac{1}{\hbar}[\gamma,-]_{\star_{MW}}.
\end{aligned}
\end{equation}
Here $[-,-]_{\star_{MW}}$ denotes the bracket associated to the fiberwise Moyal-Weyl product. In the last equality we have used the fact that the fiberwise de Rham operator $\delta$ can be expressed as a bracket with respect to the Moyal-Weyl product $\star_{MW}$. By construction, these flat connections are compatible with the fiberwise product $\star_{MW}$ on the Weyl bundle. Moreover, there is the following one-to-one correspondence:
$$
\sigma: \Gamma^{flat}(X,\W_{X,\mathbb{R}}[[\hbar]])\cong C^\infty(X)[[\hbar]].
$$ 
Here $\sigma:\W_{X,\mathbb{R}}[[\hbar]]\rightarrow C^\infty(X)[[\hbar]]$ denotes the symbol map assigning $y^i$'s in the Weyl bundle to $0$. For every smooth function $f\in C^\infty(X)[[\hbar]]$, we let $O_f$ denote the corresponding (unique) flat section with symbol $f$. 
\begin{defn}
On the formal Weyl bundle, there is the following {\em weight}
$$
|\hbar|=2, \hspace{6mm}|y^i|=1.
$$
This weight induces a filtration on the formal Weyl bundle: we let $(\W_{X,\R}[[\hbar]])_{\geq l}$ denote the sub-bundle of weights $\geq l$.
\end{defn}
From equation \eqref{equation: Moyal-Weyl-product} it is easy to see that the fiberwise Moyal-Weyl product $\star_{MW}$ is compactible with this weight.  The operators $\delta$ and $\delta^{-1}$ decreases and increases the weight by $1$ respectively. 

\begin{eg}\label{example: Fedosov-connection-flat-spaces}
On the flat space $\mathbb{R}^{2n}$ with the standard symplectic form
$$
\omega=\sum_{i=1}^n dx^{2i-1}\wedge dx^{2i},
$$
we have the following Fedosov connection:
\begin{equation}\label{equation: Fedosov-connection-R-2n}
	D=d-\delta=\sum_{i=1}^{2n}dx^i\left(\frac{\partial}{\partial x^i}-\frac{\partial}{\partial y^i}\right).
\end{equation}
In particular, the flat section $O_f$ associated to a smooth function is givens by its Taylor expansion:
\begin{equation}\label{equation: flat-section-simple-Fedosov-connection}
O_f=\sum_{l\geq 0}(\delta^{-1}\circ d)^l(f)=f+\sum_{i=1}^{2n}\frac{\partial f}{\partial x^i}y^i+\cdots.
\end{equation}
\end{eg}

\subsection{Quantum Darboux Theorem and Quantum Weierstrass Approximation}
\

It was shown in \cite{Fed1996} that every Fedosov connection can be locally written as in Example \ref{example: Fedosov-connection-flat-spaces}, known as the "Quantum Darboux Theorem". 
\begin{thm}\label{theorem: quantum-Darboux}
	Let $X$ be a symplectic manifold with a Fedosov connection $D$. Then locally on any small enough $U\subset X$, there exists a Darboux coordinate system $(x^1,\cdots, x^{2n})$ such that the connection $D$ is locally of the form as in Example \ref{example: Fedosov-connection-flat-spaces}:
	$$
	D|_U=d-\delta=\sum_{i=1}^{2n}dx^i\left(\frac{\partial}{\partial x^i}-\frac{\partial}{\partial y^i}\right).
	$$
\end{thm}
By the explicit formula in equation \eqref{equation: flat-section-simple-Fedosov-connection}, the flat sections associted to the coordinate functions $x^1,\cdots, x^{2n}$ are of the following simple form:
\begin{equation}\label{equation: flat-section-Darboux-coordinate-functions}
O_{x^i}=x^i+y^i.
\end{equation}
We call the algebra generated by these $x^i$'s under the star product {\em quantum polynomials}.
There is the following lemma describing the star product under quantum Darboux coordinates:
\begin{lem}\label{lemma: star-product-quantum-Darboux-coordinates}
	For two local formal smooth functions $f,g\in C^\infty(U)[[\hbar]]$, the quantum product $\star$ under a quantum Darboux coordinate is given by the Moyal-Weyl product:
	$$
	(\alpha\star_{MW}\beta)(x,\hbar):=\exp\left(\frac{\hbar}{2}\omega^{ij}\frac{\partial}{\partial y^i}\frac{\partial}{\partial z^j}\right)\alpha(y,\hbar)\beta(z,\hbar)\bigg|_{y=z=x}
	$$
\end{lem}
There is then the following corollary:
\begin{cor}\label{corollary: quantum-classical-polynomial}
	Let $p=p_0+\hbar\cdot p_1+\cdots+\hbar^m\cdot p_m$ denote a quantum polynomial. Then each $p_i$ is a classical polynomial. Conversely,  a classical polynomial can be written as a sum of quantum polynomials. 
\end{cor}
\begin{proof}
	The first statement follows from Lemma \ref{lemma: star-product-quantum-Darboux-coordinates}. Here we will only prove the second statement for a monomial $x^{k_1}\cdots x^{k_m}$. We will do an inductive argument on the degree $m$. For $m=0$ this is obvious. For $m>0$, we can quantize it to $x^{k_1}\star \cdots\star x^{k_m}$ whose classical part (i.e., modulo $\hbar$) is exactly the classical one. By Lemma \ref{lemma: star-product-quantum-Darboux-coordinates}, $x^{k_1}\cdots x^{k_m}-x^{k_1}\star \cdots\star x^{k_m}$ is still a polynomial (possibly with $\hbar$), whose polynomial degree in $x^1,\cdots, x^{2n}$ is strictly smaller than $m$. Then the induction hypothesis finishes the proof. 
\end{proof}

Equation \eqref{equation: flat-section-Darboux-coordinate-functions} implies that the flat section associated to a quantum polynomial must be a polynomial in both $y^i$'s and $\hbar$. We have the following generalization of quantum polynomials: 
\begin{defn}\label{definition: quantizable-functions}
	We say a formal smooth function $f$ is quantizable if its associated flat section $O_f$ lives in the finite weight sub-bundle of $\W_{X,\R}[[\hbar]]$. 
\end{defn}
From the fact that the fiberwise Moyal-Weyl product is compatible with the weight on $\W_{X,\mathbb{R}}[[\hbar]]$,  it is easy to see that quantizable functions form a subalgebra of all formal smooth functions under the formal star product. Moreover, since the star product of two quantizable functions is still quantizable (thus must be a polynomial in $\hbar$ instead of a formal power series), we can take their evaluation at any $\hbar\in\C$ to obtain a {\em non-formal quantization}. It follows from equation \eqref{equation: flat-section-Darboux-coordinate-functions}, there are the following inclusions:
\begin{equation}\label{equation: inclusions-formal-functions}
	\textit{Quantum polynomials}\subset\textit{Quantizable functions}\subset\textit{Formal smooth functions}.
\end{equation}

Recall that the classical Weierstrass Approximation theorem says that on the closed interval $[0,1]$, any continuous function can be uniformly approximated by polynomials. Equivalently, the sub-algebra of polynomials on $[0,1]$ is dense in $C^0([0,1])$ under the $C^0$-topology. The first constructive proof is given by Bernstein using a special kind of polynomials for each continuous function $f$, now known as Bernstein polynomials. It turns out that not only do the Bernstein polynomials converge to $f$, but their derivatives of any order converge to the corresponding derivatives of $f$, if they exist and are continuous.

In \cite{VV2015}, a Bernstein polynomials $\tilde{B}_n(f;x)$ is defined for every function $f\in C^m(\mathbb{R}^n)$, and the {\em uniform convergence} is proved on any $n$-dimensional cube $\mathbb{K}^k$ for any multi-index $I=(i_1,\cdots, i_n)$ with $|I|\leq m$.
$$
\tilde{B}_n^{(I)}(f;x)\rightarrow f^{(I)}(x), \hspace{5mm}x\in \mathbb{K}^d.
$$
This immediately implies the following theorem which we will need later:
\begin{thm}\label{theorem: classical-Weierstrass-theorem}
Let $f\in C^\infty(\mathbb{R}^n)$ be an infinitely differentiable function on $\mathbb{R}^n$. There exists a sequence of polynomials $\{p_i\}$, such that on any compact subset $K\subset\R^n$ there is for any multi-index $I$, 
$$
\sup\bigg\{\bigg|\frac{\partial^{|I|}(f-p_i)}{\partial x^I}(x_0)\bigg|: x_0\in K\bigg\}\rightarrow 0.
$$
\end{thm}

We have the following {\em Quantum Weierstrass Approximation} theorem, which says that quantum polynomials are locally dense in all formal smooth functions, and immediately implies that quantizable functions are locally dense by the inclusions \eqref{equation: inclusions-formal-functions}. 
\begin{thm}\label{theorem: quantum-Weierstrass}
For any quantum Darboux chart $U\subset X$, the space of quantum polynomials and quantizable functions are dense in the following sense: let $f\in C^\infty(U)[[\hbar]]$, then for any compact subset $K\subset U$ and any $\epsilon>0$, there exists a quantum polynomial $p$, such that 
$$
d_K(f,p)<\epsilon. 
$$  
(see Notation \ref{notation: metric-compact-subset} for the definition of $d_K$).
\end{thm}
\begin{proof}
We first consider the case where  $f\in C^\infty(U)$ is an ordinary smooth function on $U$. We will show that for any positive integer $N>0$, there exists a quantum polynomial $p$, such that 
\begin{equation}\label{equation: distance-K}
d_K(f,p)<\frac{1}{2^N}. 
\end{equation}
First of all, by Theorem \ref{theorem: classical-Weierstrass-theorem}, there exists a classical polynomial $\tilde{p}_0$, such that for any multi-index $I$ with $|I|\leq N+1$ there is
\begin{equation}\label{equation: norm-estimate-approximation}
\sup\bigg\{\bigg|\frac{\partial^{|I|}(f-\tilde{p}_0)}{\partial x^I}(x_0)\bigg|: x_0\in K\bigg\}<\frac{1}{(N+2)\cdot 2^{N+1}}.
\end{equation}
We can quantize $\tilde{p}_0$ as in the proof of Corollary \ref{corollary: quantum-classical-polynomial} to obtain a quantum polynomial $p_0$, whose classical part (i.e., modulor $\hbar$) coincides with $\tilde{p}_0$.

We can write the $\hbar$ power expansion of $f-p_0$ as 
$$
f-p_0=(f-\tilde{p}_0)+\hbar\cdot q_1+\cdots+\hbar^m\cdot q_m,
$$
for some $m>0$. By Lemma \ref{lemma: star-product-quantum-Darboux-coordinates}, $q_1,\cdots, q_m$ are all classical polynomials. By Corollary \ref{corollary: quantum-classical-polynomial}, $\hbar\cdot q_1+\cdots+\hbar^m\cdot q_m$ is itself a quantum polynomial which we denote by $p_1$. There is 
$$
f-p_0-p_1=f-\tilde{p}_0. 
$$
Equation \eqref{equation: norm-estimate-approximation} implies the inequality \eqref{equation: distance-K} by a simple computation. 

Next let  $f=f_0+\hbar\cdot f_1+\cdots\in C^\infty(U)[[\hbar]]$ be a general formal smooth function on $U$.  By the argument in the previous part, we can find for every $0\leq m\leq N+1$ a quantum polynomial $q_m$ such that 
$$
||\hbar^m\cdot f_m-q_m||_{K,\hbar,l}<\frac{1}{(N+2)\cdot 2^{N+1}}.
$$
for any $l\leq N+1$. 
By the triangle inequality for semi-norms, there is
\begin{align*}
	||f-\sum_{i=0}^{N+1}q_i||_{K,\hbar,l}&\leq \sum_{i=0}^{N+1}||\hbar^i\cdot f_i-q_i||_{K,\hbar, l}+\sum_{i>N+1}||\hbar^i\cdot f_i||_{K,\hbar,l}\\
	&=\sum_{i=0}^{N+1}||\hbar^i\cdot f_i-q_i||_{K,\hbar, l}\\
	&\leq \frac{N+2}{(N+2)\cdot 2^{N+1}}=\frac{1}{2^{N+1}}.
\end{align*}
In the middle equality we have used the fact that $||\hbar^i\cdot f||_{\hbar, l}=0$ if $i>N+1\geq l$.  
\begin{align*}
	d_K(f,\sum_{i=0}^{N+1}q_i)=&\sum_{l=0}^\infty\frac{1}{2^l}\cdot\frac{||f-\sum_{i=0}^{N+1}q_i||_{K,\hbar,l}}{1+||f-\sum_{i=0}^{N+1}q_i||_{K,\hbar,l}}\\
	\leq&\sum_{l=0}^{N+1}\frac{1}{2^l}\cdot\frac{||f-\sum_{i=0}^{N+1}q_i||_{K,\hbar,l}}{1+||f-\sum_{i=0}^{N+1}q_i||_{K,\hbar,l}}+\frac{1}{2^{N+1}}\\
	\leq& \frac{1}{2^{N+1}}\cdot (N+2)+\frac{1}{2^{N+1}}. 
\end{align*}
Since the last line goes to $0$ as $N\rightarrow\infty$, we obtain the desired approximation by quantum polynomials. 
\end{proof}
\begin{rmk}
	In a forthcoming paper, we will consider a similar notion of quantizable functions on K\"ahler manifolds with a quantum Weierstrass Approximation theorem, and its relation with geometric quantization. 
\end{rmk}

\subsection{Some computation in Fedosov quantization}\label{subsection: estimates-Fedosov-quantization}
\

In this subsection, we will give some estimates on the flat section associated to any smooth function. These estimates will be useful in the proof of the contunuity of the trace map in the next section. 

\begin{lem}\label{lemma: weight-I-Fedosov-connection}
	The term $I$ in the Fedosov connection live in $(\W_{X,\R}[[\hbar]])_{\geq 3}$, and the operator 
	$$
	\frac{1}{\hbar}[I,-]_{\star_{MW}}
	$$
	increases the weight by at least $1$. 
\end{lem}
\begin{proof}
	The first statement follows from the construction of the Fedosov connection. The first statement, together with the fact that the fiberwise Moyal-Weyl product preserves the weight implies the second statement. 
\end{proof}


Now we explain how the flat section $O_f$ depends on the jets of  $f$:
\begin{prop}\label{lemma: coefficients-flat-section}
	Let $f\in C^\infty(X)$ denote a smooth function on a compact  symplectic manifold $X$. Let $O_f$ denote the associated flat section of the Weyl bundle under the Fedosov connection.  For every $m\geq 0$, there exists a constant $\tilde{C}_m$ (independent of the function $f$),  such that the coefficients of all the terms in $(O_f)_m$ (in a chosen Darboux coordinate chart) is bounded above by
	\begin{equation}\label{equation: estimate-flat-section}
		\tilde{C}_m\cdot ||f||_{m}. 
	\end{equation}
	Explicitly, the constant $\tilde{C}_m$ depend on the suprema of the coefficients of the Fedosov connections (including the Christoffel symbols, curvature and the Fedosov cohomology class) and their derivatives 
\end{prop}
\begin{proof}
	Here we will do a local computation in Darboux coordinates. First of all, there is $(O_f)_0=f$, and it is shown in \cite{Fed1994} that there is the following explicit iterative formula for the flat section $O_f$:
	\begin{equation}\label{equation: iterative-formula-flat-section}
		O_f=f+\delta^{-1}\left(\nabla O_f+\frac{1}{\hbar}[I,O_f]_{\star_{MW}}\right). 
	\end{equation}
	Using the fact that the fiberwise Moyal-Weyl product $\star_{MW}$ preserves the  weight on $\W_{X,\mathbb{R}}[[\hbar]]$ and that $\delta^{-1}$ increases the weight by $1$, it is easy to see that for any $m\geq 1$, there is the following formula for the weight $m$ component $(O_f)_m$:
	\begin{equation}\label{equation: weight-components-flat-section}
		\begin{aligned}
			(O_f)_m=&\delta^{-1}\left(\nabla{(O_f)}+\frac{1}{\hbar}[I,O_f]_{\star_{MW}}\right)_{m-1}\\
			=&\delta^{-1}\left(\nabla{(O_f)_{m-1}}+\frac{1}{\hbar}[I,(O_f)_{\leq m-2}]_{\star_{MW}}\right)_{m-1}
		\end{aligned}
	\end{equation}
	Here we have used the fact that the bracket  $\frac{1}{\hbar}[I,-]_{\star_{MW}}$ increses the weight by at least $1$ in Lemma \ref{lemma: weight-I-Fedosov-connection}. 
	To show the statement about number of derivatives of $f$, we can do induction on the weight $m$. For $m=0$, this is trivial since $(O_f)_0=f$. Suppose $(O_f)_{m-1}$ contains at most $m-1$ partial derivatives of $f$, then equation \eqref{equation: weight-components-flat-section} implies the same statement for $(O_f)_m$. 
	
	
	It is then clear that each weight component $(O_f)_m$ involves the following data: partial derivatives of $f$ of order $\leq m$, and contraction with terms of $I$ in Fedosov connection of weight $\leq m$. Then this lemma follows from the following facts:
	\begin{enumerate}
		\item The monomials in $I$ of weight $m$ is bounded above by $C_m$.
		\item  Since we are choosing Darboux coordinates, the contractions between terms in the Weyl bundle are bounded above by $1$. Thus the estimate of the coefficients of $I$ and $O_f$ is enough. 
	\end{enumerate}
\end{proof}

\section{Continuity of the trace map}\label{section: continuity-trace-map}

In this section, we will let $X$ denote a {\em compact} symplectic manifold. We begin with the definition of the trace map of deformation quantization. 
\begin{defn}\label{definition: trace-map}
	Let $(C^\infty(X)[[\hbar]],\star)$ denote a formal deformation quantization of a symplectic manifold. A {\em trace} of this quantum algebra is a map:
	\begin{equation}\label{equation: trace-map}
	\Tr: C^\infty(X)[[\hbar]]\rightarrow \mathbb{C}((\hbar)),
	\end{equation}
	which satisfies the following two conditions:
	\begin{itemize}
		\item $\Tr(f)=\frac{(-1)^n}{\hbar^n}\left(\int_X f\cdot\omega^n+O(\hbar)\right)$;
		\item $\Tr(f\star g)=\Tr(g\star f)$.
	\end{itemize}
\end{defn}
It is shown in \cite{Nest-Tsygan-1995} that there exists a unique trace for each formal deformation quatization on a symplectic manifold.
There have been constructions of trace in deformation quantization via different methods by  Fedosov in \cite{Fed1994} and Nest-Tsygan in \cite{Nest-Tsygan-1995} respectively.

We will consider a “renormalized” trace for (formal) functions $C^\infty(X)[[\hbar]]$
$$
\widetilde{\Tr}:=(-\hbar)^n\cdot\Tr,
$$
 whose values live in $\mathbb{C}[[\hbar]]$ that also allows a Fr\'echet topology (Example \ref{example: sequence-complex-numbers} and \ref{example: C-hbar-Frechet-algebra}). It is then natural to ask if the trace map is {\em continous}. In this section, we will give a positive answer to this question:
\begin{thm}\label{theorem: continuity-trace-map}
The renormalized trace map $\widetilde{\text{Tr}}$ for a deformation quantization $(C^\infty(X)[[\hbar]],\star)$ in equation \eqref{equation: trace-map} is continuous with respect to the Fr\'echet topologies on both sides. 
\end{thm}

The main technique of the proof  of Theorem \ref{theorem: continuity-trace-map} is to find an appropriate cocycle in the de Rham cohomology whose integral over $X$ gives $\Tr(f)$. Here we will apply 
the Batalin-Vilkovisky (BV) quantization method for deformation quantization developed in \cite{GraLiLi2017} which gives an explicit formula for the trace map as the correlation function of the associated {\em local quantum observable} $O_f$ in a one-dimensional quantum sigma model. 
\begin{align*}
\Tr(f)=\langle O_f\rangle=&\int_X \alpha_f\cdot\left(\frac{\omega}{\hbar}\right)^n\\
=&\int_X \left(\alpha_{f,0}+\hbar\cdot \alpha_{f,1}+\hbar^2\cdot\alpha_{f,2}+\cdots\right)\cdot\omega^n\in\C[[\hbar]].
\end{align*}
We call the integrand $\alpha_f\cdot\left(\frac{\omega}{\hbar}\right)^n$ a {\em trace density} for $f$, and there is the following analysis of the trace density which implies Theorem \ref{theorem: continuity-trace-map}:
\begin{prop}\label{proposition: integrand-locality}
	Let $f\in C^\infty(X)$ be an (ordinary) smooth function on $X$.  For every $l\geq 0$ and every $x_0\in X$, $\alpha_{f,l}(x_0)$ depends on the jets of $f$ at $x_0$ of order at most $l$.  
	Equivalently, the assignment 
	$$
	f\mapsto \alpha_{f,l}
	$$
	is a differential operator of order at most $l$. In particular, there is the following inequality: for any $l\geq 0$, there exists a constant $C_l$ such that for any $f$, there is 
	$$
	\sup\{|\alpha_{f,l}(x_0)|: x_0\in X\}\leq C_l\cdot ||f||_{\hbar, l}.  
	$$
\end{prop}
\begin{proof} (of Theorem \ref{theorem: continuity-trace-map})
For any $l\geq 0$, there is the following simple inequality:
\begin{align*}
\bigg|\int_X \alpha_{f,l}\cdot\omega^n\bigg|&\leq \sup\{|\alpha_{f,l}(x_0)|: x_0\in X\}\cdot Vol_X\\
&\leq C_l\cdot ||f||_{\hbar, l}\cdot Vol_X.
\end{align*}
Here $Vol_X$ denotes the volume of $X$ under the Liouville measure. The second inequality follows from Proposition \ref{proposition: integrand-locality}. The continuity of the trace map  then follows from this comparison between semi-norms on both sides. 
\end{proof}
The proof of Proposition \ref{proposition: integrand-locality} is based on an explicit Feynman graph computation, which is an advantage of the Batalin-Vilkovisky quantization method.  In \cite{FFS1995}, an explicit cocycle for the trace is explicitly given via Hochchild chains and is essentially equivalent to the one we use here via the BV quantization method, but less convenient for the computation. 

\subsection{Trace as the correlation functions}
\

In this subsection, we prove Proposition \ref{proposition: integrand-locality} by  giving  an explicit expression of the trace density $\alpha_f\cdot\left(\frac{\omega}{\hbar}\right)^n$ in terms of Feynman graph expansions and explain its background in physics.  For more details on the underlying one-dimensional Chern-Simons model and its quantization, we refer to \cite{GraLiLi2017}.

\subsubsection{The geometry of BV bundle and quantization}
\



The BV quantization of the one-dimensional Chern-Simons model with target $X$ is described in terms of the quantum geometry of the {\em BV bundle} on $X$. 
\begin{defn}\label{definition: BV-bundle}
	The BV bundle of a symplectic manifold is defined as follows:
	\begin{equation}\label{equation: BV-bundle}
		\widehat{\Omega}_{TX}^{-\bullet}:=\widehat{\Sym}(T^*X)\otimes\wedge^{-\bullet}T^*X. 
	\end{equation}
	Here $\wedge^{-\bullet}T^*X:=\bigoplus_k\wedge^k T^*X[k]$. In particular, each component $\wedge^kT^*X$ sits in cohomological degree $-k$.  We will also consider the extension of the BV bundle by the differential forms, i.e., $\A_X^{\bullet}\left(\widehat{\Omega}_{TX}^{-\bullet}\right)$.
\end{defn}
In short, we add the fermionic component $\wedge^{-\bullet}T^*X$ to the Weyl bundle to obtain the BV bundle. Equivalently,  $\widehat{\Omega}_{TX}^{-\bullet}$ can be viewed as the relative de Rham complex for the bundle $TX\rightarrow X$ formally completed at the zero section, with a fiberwise de Rham operator $d_{TM}$ mapping a (local) generator $y^i$ of $\W_{X,\R}$ (see equation \eqref{equation: Weyl-bundle-monomial}) to a generator of $\wedge^{-\bullet}T^*X$ which we denote by $dy^i$. We will call $y^i$'s and $dy^i$'s {\em bosons} and {\em fermions} respectively.

 The symplectic form $\omega$ induces a canonical {\em top degree fermion} on the BV bundle of cohomological degree $-2n$ which we denote by $\widetilde{\omega}^n:=(\omega_{ij}dy^i\wedge dy^j)^n$. There is the following {\em Berezin integral} for sections of the BV bundle denoted by $\int_{Ber}$, which takes the coefficients of the “top degree fermion” in the BV bundle followed by taking the symbol of the Weyl bundle:
\begin{equation}\label{equation: BV-integration}
\int_{Ber}:=\sigma\circ\left(\hbar\cdot\iota_\Pi\right)^n: \A_X^{\bullet}\left(\widehat{\Omega}_{TX}\right)[[\hbar]]\rightarrow \A_X^\bullet[[\hbar]].
\end{equation}
Here $\Pi=\omega^{-1}\in\Gamma(X,\wedge^2 TX)$ denotes the Poisson tensor, and $\iota_\Pi: \widehat{\Omega}_{TX}^{-\bullet}\rightarrow 	\widehat{\Omega}_{TX}^{-\bullet+2}$ denotes the fiberwise contraction. 

For a smooth function $f\in C^\infty(X)$,  the formal smooth function $\alpha_f$ in the {\em trace density} associated to  $f$ can be explicitly written as a Berezin integral:
\begin{equation}\label{equation: alpha-f}
\alpha_f\cdot\omega^n=\int_{Ber}[O_f]_\infty\cdot e^{\gamma_\infty/\hbar}.
\end{equation}
The section  $[O_f]_\infty\cdot e^{\gamma_\infty/\hbar}$ of  the {\em BV} bundle of $X$ is obtained by applying the  {\em Renormalization group flow} operator to $O_f$ in the one-dimensional Chern-Simons model, described in the following explicit formula: 
\begin{equation}\label{equation: RG-flow-quantum-observable}
[O]_\infty e^{\gamma_\infty/\hbar}:=\text{Mult}\int_{S^1[*]}e^{\hbar\partial_P+D}(Od\theta_1)\otimes e^{\otimes\gamma/\hbar}.
\end{equation}
We will briefly describe the notations in equation \eqref{equation: RG-flow-quantum-observable} and refer to section 2.4 of \cite{GraLiLi2017} for a detailed exposition. For each $m\geq 0$, $S^1[m]$ denotes the compactified configuration space of $m$-ordered points on the circle $S^1$ constructed via successive real-oriented blow ups of $(S^1)^m$, and is a manifold with corners.  For instance, when $m=2$, there is the following explicit parametrization of $S^1[2]$ as a cylinder:
\begin{equation}\label{equation: propagator}
S^1[2]=\{(e^{2\pi i\theta},u): 0\leq\theta<1, 0\leq u\leq 1\}.
\end{equation}

We will need the following maps for configuration spaces to understand equation \eqref{equation: RG-flow-quantum-observable}. The first one is the blow down map for all $m\geq 1$:
\begin{equation}\label{equation: blow-down-maps}
\pi_m: S^1[m]\rightarrow (S^1)^m.
\end{equation}
There is also a natural projection for every subset $V\subset\{1,\cdots, m\}$:
\begin{equation}\label{equation: projection-map-configuration-space}
\pi_V: S^1[m]\rightarrow S^1[|V|]. 
\end{equation}

The term $\gamma$ in equation \eqref{equation: RG-flow-quantum-observable} is the same as in the Fedosov connection \eqref{equation: Fedosov-connection}, and the operator $D$ is explicitly given by 
\begin{equation}\label{equation: operator-D}
\begin{aligned}
D: \A^\bullet_{S^1[k]}\otimes_{\mathbb{R}}\A_X^\bullet(\widehat{\Omega}^{-\bullet}_{TX})^{\otimes k}\rightarrow&\A^\bullet_{S^1[k]}\otimes_{\mathbb{R}}\A_X^\bullet(\widehat{\Omega}^{-\bullet}_{TX})^{\otimes k} \\
 D(a_1\otimes\cdots\otimes a_k):=&\sum_{\alpha}\pm d\theta^\alpha\otimes_{\mathbb{R}}(a_1\otimes\cdots\otimes d_{TX}a_\alpha\otimes\cdots\otimes a_k).
\end{aligned}
\end{equation}
Now we finish the explanation of the right hand side of equation \eqref{equation: RG-flow-quantum-observable}. For every $m\geq 1$, we obtain
$$
\left(d\theta_1\otimes d_{TX}(\gamma)\right)\otimes\cdots\otimes\left(d\theta_m\otimes d_{TX}(\gamma)\right)\in \A_{(S^1)^m}^m\otimes\A_X^m\left(\widehat{\Omega}^{-1}_{TX}\right)^{\otimes m}. 
$$
 This will be pulled back to $\A_{S^1[m]}^m\otimes\A_X^m\left(\widehat{\Omega}^{-1}_{TX}\right)^{\otimes m}$ via the blow down map \eqref{equation: blow-down-maps}. (Notice that for every $\alpha\in\{1,\cdots,m\}$ there is a copy of the BV bundle $\widehat{\Omega}^{-\bullet}_{TX}$, the map $\text{Mult}$ takes the natural super commutative product to combine these copies of BV bundle to a single one). The exponential term in equation \eqref{equation: RG-flow-quantum-observable} is the following formal sum:
 $$
 e^{\otimes \gamma/\hbar}=\sum_{m\geq 0}\frac{1}{m!\cdot\hbar^m}\gamma^{\otimes m},\hspace{5mm}\gamma^{\otimes m}\in \A_X^{\bullet}(\widehat{\Omega}_{TX}^{-\bullet})^{\otimes m}.
 $$  
 For every two point subset $\{\alpha_1,\alpha_2\}\subset\{1,\cdots, m\}$, the propagator $\partial_P$ in equation \eqref{equation: RG-flow-quantum-observable} is defined as the following operator on $\A_{(S^1)^m}^m\otimes\left(\A_X^1(\widehat{\Omega}^{-1}_{TX})\right)^{\otimes m}$.  
 $$
 \partial_P:=\pi_{\{\alpha_1,\alpha_2\}}^*\left(u-\frac{1}{2}\right)\left(\omega^{ij}\partial_{y^i}\otimes\partial_{y^j}\right). 
 $$
 Here $\pi_{\{\alpha_1,\alpha_2\}}$ is the projection map in equation \eqref{equation: projection-map-configuration-space}, and the bi-differential $\omega^{ij}\partial_{y^i}\otimes\partial_{y^j}$ acts on the two components of the BV bundle corresponding to $\alpha_1$ and $\alpha_2$ respectively. Finally, the composition of integral over $S^1[m]$ and $\textit{Mult}$ gives rise to a section of the BV bundle $\A_X^{\bullet}\left(\widehat{\Omega}_{TX}\right)((\hbar))$.

 Equation \eqref{equation: RG-flow-quantum-observable} is the model of path integrals  in the one-dimensional Chern-Simons model, which can also be explicitly computed via Feynman graph expansions. This Chern-Simons model is a mathematical description of a sigma model with target the symplectic manifold $X$ and domain $S^1$, which explains the appearance of integrals over (compactified) configuration spaces of $S^1$.
 


\subsubsection{A graphical description of equation \eqref{equation: RG-flow-quantum-observable}}
\

In this subsection, we give an explicit expression of the BV bundle section $[O_f]_\infty\cdot e^{\gamma_\infty/\hbar}$ via Feynman graph expansions. We first recall the definition of graphs here, for more details we refer to \cite{Costello-book} and \cite{GraLiLi2017}.
\begin{defn}
	A graph $\mathcal{G}$ consists of the following data:
	\begin{enumerate}
		\item A finite set of vertices $V(\mathcal{G})$;
		\item A finite set of half-edges $H(\mathcal{G})$;
		\item An involution $\sigma: H(\mathcal{G})\rightarrow H(\mathcal{G})$. The set of fixed points of this map is denoted by $T(\mathcal{G})$ and is called the set of tails of $\mathcal{G}$. The set of two-element orbits is denoted by $E(\mathcal{G})$ and is called the set of internal edges of $\mathcal{G}$;
		\item A map $\pi: H(\mathcal{G})\rightarrow V(\mathcal{G})$ sending a half-edge to the vertex to which it is attached;
		\item A map $g: V(\mathcal{G})\rightarrow\mathbb{Z}_{\geq 0}$ assigning a genus to each vertex. 
	\end{enumerate}
\end{defn}

It is easy to see how to construct a topological space $|\mathcal{G}|$ from the above abstract data. 
A relevant {\em connected} Feynman graph is as in the following picture: 
 \begin{equation}\label{equation: graph}
 \figbox{0.26}{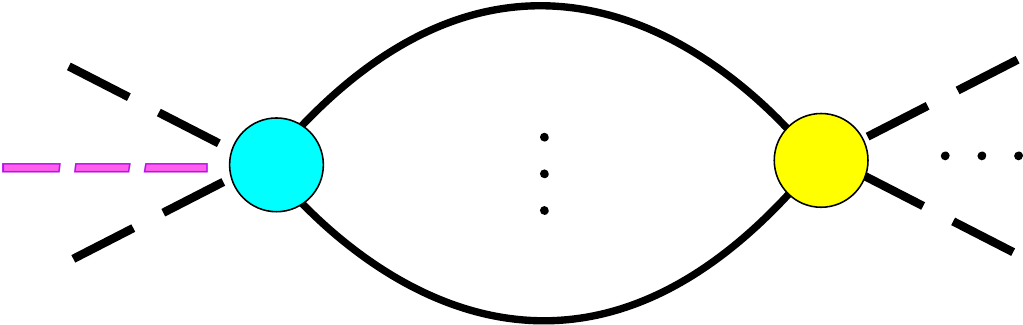}
 \end{equation}

Now we interprete equation \eqref{equation: RG-flow-quantum-observable} as a sum of {\em Feynman weights} of graphs. First of all, we label graphs by objects in quantum geometry by the following rules.
\begin{enumerate}
	\item Exactly one vertex is labeled by $d\theta\otimes O_f$ (of yellow color in graph \eqref{equation: graph}) associated to a smooth function $f$. All the other vertices are labeled by the term $D(\gamma/\hbar)$ in the Fedosov connection (of green color in graph \eqref{equation: graph}). If we write $\gamma$ in the power expansion of $\hbar$ as:
	$$
	\gamma=\gamma_0+\hbar\cdot\gamma_1+\hbar^2\cdot\gamma_2+\cdots.
	$$
	Then each blue vertex of internal genus $k$ and valency $m$ is labeled by the term in $\hbar^k\cdot D(\gamma_k/\hbar)$ of polynomial degree $l$. Similarly, we also have a genus of the (unique) yellow vertex and its label takes the term in $d\theta\otimes O_f$ of the corresponding $\hbar$ power. 
		\item For every vertex $v\in V(\mathcal{G})$, the number of half-edges connecting to it is called the {\em valency} of $v$. To each vertex of valency $k$, we label it with the term in $D(\gamma)$ or $O_f$ of polynomial degree $k$ (in both $\widehat{\Sym}T^*X$ and $\wedge^{-\bullet}T^*X$) in the BV bundle.
	\item There are two types of half-edges that are allowed in our graphs. The black half-edges will be labeled by the boson inputs (generators of $\widehat{\Sym}T^*X$), and the purple half-edges will be labeled by the fermion inputs (generators of $\wedge^{-\bullet}T^*X$). All these half-edges together form sections of $\widehat{\Omega}_{TX}$. In particular, an internal edge can only be the union of two black half-edges. 
	\item The internal edges (solid lines) of this graph are labeled by the propagator $\hbar\cdot\partial_P$.
\end{enumerate}
We summarize the labeling in the following table:
\begin{center}
	\renewcommand{\arraystretch}{1.2}
	\begin{tabular}{|c|c|}
		\hline
		Feynman graphs & Quantum geometry\\
		\hline
		Green vertex & $D(\gamma/\hbar)$\\
		\hline
		Yellow vertex & $d\theta\otimes O_f$\\
		\hline
		Genus of a vertex & Power of $\hbar$\\
		\hline
		Valency of a vertex & Polynomial degree in $\widehat{\Omega}^{-\bullet}_{TX}$\\
		\hline
		Internal edge & Propagator $\hbar\cdot\partial_P$\\
		\hline
		Black half edge & Generator of $\widehat{\Sym}(T^*X)$ (boson)\\
		\hline
		Purple half edge & Generator of $\wedge^{-\bullet}T^*X$ (fermion)\\
		\hline
		Valency of a vertex & Polynomial degree in $\widehat{\Omega}^{-\bullet}_{TX}$\\
		\hline
	\end{tabular}
\end{center}
Next, we describe the Feynman weights of a labeled graph $\mathcal{G}$, which takes value in the BV bundle $\A_X^\bullet\left(\widehat{\Omega}^{-\bullet}_{TX}\right)$. Let $m=|V(\mathcal{G})|$ denote the number of vertices in $\mathcal{G}$.  The Feynman weight of $\mathcal{G}$ is given by contractions on the BV bundle followed by an integral over the configuration space $S^1[m]$.  

Explicitly, for every vertex $v\in V(\mathcal{G})$, it corresponds to exactly one copy of $S^1$ in $(S^1)^m$. (In particular, we let the unique yellow vertex correspond to the first copy and this is the origin of the notation $d\theta_1$ in equation \eqref{equation: RG-flow-quantum-observable}). For the green vertices, the operator $D$ turns $\gamma$ into a $1$-form on the corresponding $S^1$ component. The blow down map
$$
\pi_m: S^1[m]\rightarrow (S^1)^m
$$
allows us to pull back these forms to $S^1[m]$.

The half edges attached to  $v$ are graphical representation of sections of the corresponding copy of the BV bundle via the labeling (either $D(\gamma)$ or $O_f$). Among these half edges, those attached to other vertices to become an internal edge (solid lines) will be contracted with $\hbar\cdot\partial_P$ labeling this edge, and those to become tails (dotted lines) will remain to be generators of BV bundle in the output Feynman weight.  The symbol map $\sigma$ in equation \eqref{equation: BV-integration} is then equivalent to picking those graphs without any black tails. Then the integration $\int_{S^1[m]}$ gives a section of the BV bundle $\A_X^{\bullet}\left(\widehat{\Omega}_{TX}^{-\bullet}\right)((\hbar))$.



\subsection{The proof of Proposition \ref{proposition: integrand-locality}}
\

There are the following simple observations for each green vertex (labeled by $D(\gamma/\hbar)$): it is a $1$-form on the target $X$, and the operator $D$ in equation \eqref{equation: operator-D} has two effects: on one hand, it turns one $y^i$ in $\gamma/\hbar$ to $dy^i$ (i.e., the fiberwise de Rham differential). And on the other hand, it contributes the canonical volume form on the corresponding $S^1$.  For the unique yellow vertex, it is a $0$-form on the target and does not contain the $\wedge^{-\bullet}T^*X$ component of the BV bundle.  In the following lemma, we will describe those labeled graphs whose Feynman weight will contribute non-trivially after the BV integration $\int_{Ber}$ (we call these graphs {\em admissible}):
\begin{lem}\label{lemma: number-of-vertices}
The trace density  $\alpha_f\cdot\left(\frac{\omega}{\hbar}\right)^n$ associated to a function $f$ in equation \eqref{equation: alpha-f} is a linear combination of the weights of (not necessarily connected) Feynman graphs with exactly $1$ yellow vertex (labeled by $O_f$) and $2n$ green vertices (labeled by $D(\gamma/\hbar)$). 
\end{lem}
\begin{proof}
	We have seen in the previous discussion  that each green vertex contributes exactly a section in $\A_X^1\left(\widehat{\Sym}(T^*X)\otimes\wedge^{-1}T^*X\right)((\hbar))\subset\A_X^{\bullet}\left(\widehat{\Omega}_{TX}^{-\bullet}\right)((\hbar))$, while a yellow vertex does not contain components in $\A_X^\bullet$ and $\wedge^{-\bullet}T^*X$. This lemma then follows from the fact that only terms in $\A_X^{2n}\left(\widehat{\Sym}{T^*X}\otimes\wedge^{-2n}T^*X\right)((\hbar))\subset\A_X^{\bullet}\left(\widehat{\Omega}_{TX}^{-\bullet}\right)((\hbar))$ could contribute non-trivially to the BV integration and the integral over $X$. 
\end{proof}

 Suppose the weight of an admissible graph $\mathcal{G}$ contributes to $\hbar^l\cdot\alpha_l\cdot\left(\frac{\omega}{\hbar}\right)^n$ after the BV integration.  Let $p\geq 0$ be the sum of the genus of these $2n$ green vertices in this graph, and let $g$ denote the genus of the yellow vertex labeled by $d\theta_1\otimes O_f$. Then the power of $\hbar$ contributed by these vertices are $p-2n+g$ (there is the $-2n$ here since each green vertex is labeled by the correct component in $D(\gamma/\hbar)$). We have seen that such a graph must contain exactly $2n$ purple tails (fermions), and all the black half edges (bosons) must be contained in the set of internal edges (otherwise the boson output will be annihilated by the symbol map $\sigma$ in equation \eqref{equation: BV-integration}).  By comparing the power of $\hbar$ in $\hbar^l\cdot\alpha_l\cdot\left(\frac{\omega}{\hbar}\right)^n$ and the corresponding Feynman weights, we obtain the following equality:
 $$
 l-n=p-2n+g+|E(\mathcal{G})|+n.
 $$
 The term $|E(\mathcal{G})|$ (i.e., number of internal edges in $\mathcal{G}$) in the right hand side arises since the labeling of each internal edge by $\hbar\cdot\partial_P$ will contribute an $\hbar$. The right most term $n$ comes from the operator $(\hbar\cdot\iota_\Pi)^n$. It follows that 
 $$
 |E(\mathcal{G})|=l-p-g\leq l. 
 $$
 Since the valency of the unique yellow vertex (labeled by $d\theta_1\otimes O_f$) can not exceed $E(\mathcal{G})$, it is bounded above by $l$. Thus we can only label it with $d\theta_1\otimes (O_f)_m$ for $m\leq l$. In the proof of Proposition \ref{lemma: coefficients-flat-section} we have seen that the coefficients of $(O_f)_m$ at any $x_0\in X$ depends on the Taylor expansion of $f$ at $x_0$ of order at most $m$. This finishes the proof of Proposition \ref{proposition: integrand-locality}. 


\appendix

\bibliographystyle{amsplain}
\bibliography{References}

\begin{bibdiv}
\begin{biblist}

\bib{Beiser-Romer-Waldmann}{article}{
      author={Beiser, S.},
      author={R\"omer, H.},
      author={Waldmann, S.},
       title={Convergence of the {W}ick star product},
        date={2007},
        ISSN={0010-3616,1432-0916},
     journal={Comm. Math. Phys.},
      volume={272},
      number={1},
       pages={25\ndash 52},
}

\bib{Beiser-Waldmann}{article}{
      author={Beiser, S.},
      author={Waldmann, S.},
       title={Fr\'echet algebraic deformation quantization of the
  {P}oincar\'e{} disk},
        date={2014},
        ISSN={0075-4102,1435-5345},
     journal={J. Reine Angew. Math.},
      volume={688},
       pages={147\ndash 207},
}

\bib{ChaLeuLi2023}{article}{
      author={Chan, K.},
      author={Leung, N.C.},
      author={Li, Q.},
       title={Quantizable functions on {K}\"{a}hler manifolds and non-formal
  quantization},
        date={2023},
        ISSN={0001-8708},
     journal={Adv. Math.},
      volume={433},
       pages={Paper No. 109293},
}

\bib{Costello-book}{book}{
      author={Costello, K.},
       title={Renormalization and effective field theory},
      series={Mathematical Surveys and Monographs},
   publisher={American Mathematical Society, Providence, RI},
        date={2011},
      volume={170},
        ISBN={978-0-8218-5288-0},
}

\bib{Fed1994}{article}{
      author={Fedosov, B.~V.},
       title={A simple geometrical construction of deformation quantization},
        date={1994},
        ISSN={0022-040X},
     journal={J. Differential Geom.},
      volume={40},
      number={2},
       pages={213\ndash 238},
}

\bib{Fed1996}{book}{
      author={Fedosov, B.~V.},
       title={Deformation quantization and index theory},
      series={Mathematical Topics},
   publisher={Akademie Verlag, Berlin},
        date={1996},
      volume={9},
        ISBN={3-05-501716-1},
}

\bib{Fed2000}{article}{
      author={Fedosov, B.~V.},
       title={The {A}tiyah-{B}ott-{P}atodi method in deformation quantization},
        date={2000},
        ISSN={0010-3616},
     journal={Comm. Math. Phys.},
      volume={209},
      number={3},
       pages={691\ndash 728},
}

\bib{FFS1995}{article}{
      author={Feigin, B.},
      author={Felder, G.},
      author={Shoikhet, B.},
       title={Hochschild cohomology of the {W}eyl algebra and traces in
  deformation quantization},
        date={2005},
        ISSN={0012-7094,1547-7398},
     journal={Duke Math. J.},
      volume={127},
      number={3},
       pages={487\ndash 517},
}

\bib{GraLiLi2017}{article}{
      author={Grady, R.~E.},
      author={Li, Q.},
      author={Li, S.},
       title={Batalin-{V}ilkovisky quantization and the algebraic index},
        date={2017},
        ISSN={0001-8708},
     journal={Adv. Math.},
      volume={317},
       pages={575\ndash 639},
}

\bib{Nest-Tsygan-1995}{article}{
      author={Nest, R.},
      author={Tsygan, B.},
       title={Algebraic index theorem},
        date={1995},
        ISSN={0010-3616,1432-0916},
     journal={Comm. Math. Phys.},
      volume={172},
      number={2},
       pages={223\ndash 262},
}

\bib{OMMY}{incollection}{
      author={Omori, H.},
      author={Maeda, Y.},
      author={Miyazaki, N.},
      author={Yoshioka, A.},
       title={Deformation quantization of {F}r\'echet-{P}oisson algebras:
  convergence of the {M}oyal product},
        date={2000},
   booktitle={Conf\'erence {M}osh\'e{} {F}lato 1999, {V}ol. {II} ({D}ijon)},
      series={Math. Phys. Stud.},
      volume={22},
   publisher={Kluwer Acad. Publ., Dordrecht},
       pages={233\ndash 245},
}

\bib{Rieffel1989}{article}{
      author={Rieffel, M.A.},
       title={Deformation quantization of {H}eisenberg manifolds},
        date={1989},
        ISSN={0010-3616,1432-0916},
     journal={Comm. Math. Phys.},
      volume={122},
      number={4},
       pages={531\ndash 562},
}

\bib{VV2015}{article}{
      author={Veretennikov, A.~Yu.},
      author={Veretennikova, E.~V.},
       title={On partial derivatives of multivariant {B}ernstein polynomials},
        date={2016},
     journal={Siberian Adv. Math.},
      volume={26},
      number={4},
       pages={294\ndash 305},
}

\end{biblist}
\end{bibdiv}

\end{document}